\documentclass[notitlepage,11pt,reqno]{amsart}
\usepackage[foot]{amsaddr}
\usepackage[numbers,sort]{natbib}

\usepackage{amssymb,nicefrac,bm,upgreek,mathtools,verbatim}
\usepackage[final]{hyperref}
\usepackage[mathscr]{eucal}
\usepackage{dsfont}

\usepackage{color}
\definecolor{dmagenta}{rgb}{.4,.1,.5}
\definecolor{dblue}{rgb}{.0,.0,.5}
\definecolor{mblue}{rgb}{.0,.0,.8}
\definecolor{ddblue}{rgb}{.0,.0,.4}
\definecolor{dred}{rgb}{.6,.0,.0}
\definecolor{dgreen}{rgb}{.0,.5,.0}
\definecolor{Eeom}{rgb}{.0,.0,.5}

\usepackage[normalem]{ulem}

\usepackage[margin=1in]{geometry}
\allowdisplaybreaks

\newtheorem{lemma}{Lemma}[section]
\newtheorem{theorem}{Theorem}[section]

\newtheorem{corollary}{Corollary}[section]

\theoremstyle{definition}

\theoremstyle{remark}

\newtheorem{remark}{Remark}[section]

\numberwithin{equation}{section}

\hypersetup{
  colorlinks=true,
  citecolor=dblue,
  linkcolor=dblue,
  frenchlinks=false,
  pdfborder={0 0 0},
  naturalnames=false,
  hypertexnames=false,
  breaklinks}
\usepackage[capitalize,nameinlink]{cleveref}

\crefname{section}{Section}{Sections}
\crefname{subsection}{Subsection}{Subsections}
\crefname{condition}{Condition}{Conditions}
\crefname{hypothesis}{Hypothesis}{Conditions}
\crefname{assumption}{Assumption}{Assumptions}
\crefname{lemma}{Lemma}{Lemmas}
\crefname{claim}{Claim}{Claims}

\Crefname{figure}{Figure}{Figures}

\crefformat{equation}{\textup{#2(#1)#3}}
\crefrangeformat{equation}{\textup{#3(#1)#4--#5(#2)#6}}
\crefmultiformat{equation}{\textup{#2(#1)#3}}{ and \textup{#2(#1)#3}}
{, \textup{#2(#1)#3}}{, and \textup{#2(#1)#3}}
\crefrangemultiformat{equation}{\textup{#3(#1)#4--#5(#2)#6}}%
{ and \textup{#3(#1)#4--#5(#2)#6}}{, \textup{#3(#1)#4--#5(#2)#6}}%
{, and \textup{#3(#1)#4--#5(#2)#6}}

\Crefformat{equation}{#2Equation~\textup{(#1)}#3}
\Crefrangeformat{equation}{Equations~\textup{#3(#1)#4--#5(#2)#6}}
\Crefmultiformat{equation}{Equations~\textup{#2(#1)#3}}{ and \textup{#2(#1)#3}}
{, \textup{#2(#1)#3}}{, and \textup{#2(#1)#3}}
\Crefrangemultiformat{equation}{Equations~\textup{#3(#1)#4--#5(#2)#6}}%
{ and \textup{#3(#1)#4--#5(#2)#6}}{, \textup{#3(#1)#4--#5(#2)#6}}%
{, and \textup{#3(#1)#4--#5(#2)#6}}

\crefdefaultlabelformat{#2\textup{#1}#3}

\makeatletter
\DeclareRobustCommand\widecheck[1]{{\mathpalette\@widecheck{#1}}}
\def\@widecheck#1#2{%
    \setbox\z@\hbox{\m@th$#1#2$}%
    \setbox\tw@\hbox{\m@th$#1%
       \widehat{%
          \vrule\@width\z@\@height\ht\z@
          \vrule\@height\z@\@width\wd\z@}$}%
    \dp\tw@-\ht\z@
    \@tempdima\ht\z@ \advance\@tempdima2\ht\tw@ \divide\@tempdima\thr@@
    \setbox\tw@\hbox{%
       \raise\@tempdima\hbox{\scalebox{1}[-1]{\lower\@tempdima\box
\tw@}}}%
    {\ooalign{\box\tw@ \cr \box\z@}}}
\makeatother

\DeclareMathOperator{\Prob}{\mathbb{P}} 
\newcommand{\D}{\mathrm{d}}          
\newcommand{\RR}{\mathbb{R}}         
\newcommand{\Rd}{{\mathbb{R}^d}}       

\newcommand{\sorder}{{\mathfrak{o}}} 

\newcommand{\cL}{\mathcal{L}}        

\newcommand{\cB}{\mathcal{B}}    
\newcommand{\sB}{\mathscr{B}}    
\newcommand{\cC}{\mathcal{C}}     
\newcommand{\cD}{\Omega}     

\newcommand{\sF}{\mathfrak{F}}    
\newcommand{\sG}{\mathscr{G}}    
\newcommand{\bI}{\mathbf{I}}
\newcommand{\cI}{\mathcal{I}}    
\newcommand{\cK}{\mathcal{K}}     
\newcommand{\sM}{\mathscr{M}}     
\newcommand{\sP}{\mathscr{P}}    
\newcommand{\cS}{\mathcal{S}}   
\newcommand{\cT}{\mathcal{T}}    
\newcommand{\cX}{\mathcal{X}}    

\newcommand{\abs}[1]{\lvert#1\rvert}
\newcommand{\norm}[1]{\lVert#1\rVert}

\DeclareMathOperator{\dist}{dist}
\DeclareMathOperator{\Deg}{\mathrm{deg}}

\DeclareMathOperator{\rin}{r_{\mathrm{in}}}
\DeclareMathOperator{\rout}{r_{\mathrm{out}}}
\begin{document}

\title[Principal eigenvalue of integro-differential operators]%
{Principal eigenvalues of a class of nonlinear integro-differential operators}

\author[Anup Biswas]{Anup Biswas$^\dag$}
\address{$^\dag$ Department of Mathematics,
Indian Institute of Science Education and Research,
Dr. Homi Bhabha Road, Pune 411008, India}
\email{anup@iiserpune.ac.in}

\date{}

\begin{abstract}
We consider a class of nonlinear integro-differential operators and prove existence of two principal (half) eigenvalues in bounded smooth domains with 
exterior Dirichlet condition. We then establish simplicity of the principal eigenfunctions in viscosity sense,
 maximum principles, continuity property of the principal eigenvalues with respect to domains etc. We also prove
an anti-maximum principle and study existence result for some nonlinear problem via Rabinowitz bifurcation-type results.
\end{abstract}



\maketitle

\section{Introduction and main results}

Given a smooth bounded domain $\cD$ we consider the nonlinear integro-differential equation
$$\bI u \;=\; f \quad \text{in}\; \cD, \quad \text{and}\quad u\;=\; 0\quad \text{in}\; \cD^c,$$
where $\bI$ is suprema of linear operators in $\cL_*$ (see \eqref{E1.1} and \eqref{Assmp1} below) with nonsmooth coefficients.
The main theme of this paper is to study the associated eigenvalue problems, existence and uniqueness results, maximum principles,
anti-maximum principles etc. We also study existence result for certain nonlinear problems involving nonlocal Pucci's operators via Rabinowitz bifurcation-type results.

 There is a well established theory for eigenvalue problems in the
literature of local partial differential equations. See for instance, \cite{A08, B77, BEQ, BNV, FQ, L83, QS08} and references therein. Eigenvalue
problems for the above type of nonlinear elliptic operators were first considered in \cite{P66} where existence of two principal eigenvalues
were shown. These eigenvalues are also referred to as \textit{half-eigenvalues} or \textit{demi-eigenvalues}. In the celebrated work \cite{BNV} an interesting
connection between (refined) maximum principle and principal eigenvalues of linear elliptic operators were established. For eigenvalue problems of
fully nonlinear elliptic operators we refer \cite{A08, FQ, IY, L83, QS08}. Surprisingly, there are not many existing works on eigenvalue problems
involving stable like operators. In \cite{BCH,C10, KMPZ} eigenvalue problems are studied for a class of nonlocal linear operators. The nonlocal kernels
appeared in \cite{BCH,C10, KMPZ} are different from ours. Recently in \cite{BL17a, BL17b}, the authors consider eigenvalue problems and 
study related maximum principles
for a large class of linear nonlocal schr\"{o}dinger operators arising from subordinate Brownian motions. The central object in the analysis of
\cite{BL17a, BL17b} is the Feynman-Kac representation of the eigenfunctions. Another recent work \cite{QSX} considers the eigenvalue problem for nonlinear 
integro-differential operators with the drift term and establish existence of principal eigenvalues.

During last decade there has been substantial research devoted to the
development of the regularity theory of nonlinear integro-differential operators \cite{CS09, CS11, K13, RS16, RS17, S15}. These results
are main ingredients to apply the nonlinear Krein-Rutman theorem which is used below to find the principal eigenvalues of our
nonlinear operator. We also provide characterizations of principal eigenvalues and eigenfunctions using viscosity solutions. In fact, we are able to
produce most of the results of \cite{BEQ} which studies eigenvalue problems for (local) Pucci's operators. 
We remark that eigenvalue problem is an important tool in the study of solutions at resonance, Ladezman-Lazer type results and Ambrosetti- Prodi phenomenon.

Rest of the article is organized as follows. In the  next section we describe our model and the basic assumptions. Section~\ref{S-main} describes
the main results of this article and Section~\ref{S-moti} supplies a motivation for considering these class of integro-differential operators.
In Section~\ref{S-prelim} we gather some preliminary results, mostly from literature and then in Section~\ref{S-proofs} we provide
the proofs of our main results.

\subsection{Assumptions on the model}

The ellipticity class is defined with respect to a class of operators $\cL$ containing operator $L$ of the form
\begin{equation}\label{E1.1}
L u(x) = \frac{1}{2}\int_{\Rd}\bigl(u(x+y)+u(x-y)-2u(x)\bigr) \frac{k(y/\abs{y})}{\abs{y}^{d+2s}}\, \D{y},\quad s\in(0, 1)\;  \text{fixed},
\end{equation}
where for some fixed $\lambda, \Lambda, 0<\lambda\leq\Lambda$, it holds that
\begin{equation}\label{Assmp1}
\lambda\leq k\leq \Lambda, \quad \text{and}\quad k(y)=k(-y). \tag{A}
\end{equation}
Therefore, $k\in L^\infty(S^{d-1})$. We would be interested in the operator
\begin{equation}\label{Opt}
\bI u = \sup_{L\in\cL} \, Lu.
\end{equation}
Let $\cL_*$ be the class of all operators satisfying \eqref{Assmp1}. Trivially, $\cL\subset\cL_*$. The maximal and minimal operator with respect to $\cL_*$ is defined as 
follows.
$$\sM^+_* u =\sup_{L\in \cL_*}\, L u, \quad \text{and}, \quad \sM^-_* u= \inf_{L\in\cL_*}\, Lu.$$
Let us also introduce a larger class of operators $\cL_0$. $\cL_0$ contains operators of the form \eqref{E1.1} where the kernels are given by functions of the form $k(y) \abs{y}^{-d-2s}$ satisfying
$k(y)=k(-y)$ and $\lambda\leq k\leq \Lambda$. The extremal operators corresponding to $\cL_0$ are given by
$$\sM^+_0 u = \; \sup_{L\in\cL_0} Lu, \quad \text{and}, \quad \sM^-_0 u = \; \inf_{L\in\cL_0}\; Lu.$$
It is helpful to notice that $\sM^-_0\leq \sM^-_*\leq \bI\leq \sM^+_*\leq \sM^+_0$. Therefore, if a operator is elliptic with respect to the class $\cL$,
 in the sense of \cite[p.~603]{CS09}, it also
elliptic with respect to the class $\cL_*$ and $\cL_0$.

Let $\cD$ be a  bounded domain with smooth boundary.
We study the Dirichlet eigenvalue problem for the nonlinear operator $\bI$ in $\cD$. To do so we apply the nonlinear
Krein-Rutman theorem and therefore, it is important to know the \textit{exact} boundary behaviour of the solutions, vanishing on $\cD^c$, near the boundary. It is recently shown in 
\cite[Section~2.1]{RS16} that in general, it is not possible to have a fine boundary regularity for the operators elliptic with respect to the class $\cL_0$. 
Therefore, we restrict the ellipticity class to subsets of $\cL_*$. 


Let us now define the Dirichlet principal values for the operator $\bI$. By a (sub or super) solution  we always mean viscosity (sub or super, resp.) solution . See for instance, \cite{CS09} for
definition and properties of viscosity solutions of nonlocal operators. Let $\omega(y)=(1+\abs{y})^{-d-2s}$ and $L^1(\omega)$ be the set of all integrable functions
with respect to the weight function $\omega$.
For any $\mu\in\RR$ we define
\begin{align*}
\sF^+(\cD, \mu) &=\{ \psi\in\cC(\Rd)\cap L^1(\omega)\; :\; \psi\geq 0,\; \psi>0 \; \text{in}\; \cD, \; \bI \psi + \mu\psi \leq 0 \; \text{in}\; \cD\},
\\
\sF^-(\cD, \mu) &=\{ \psi\in\cC(\Rd)\cap L^1(\omega)\; :\; \psi\leq 0,\; \psi<0 \; \text{in}\; \cD, \; \bI \psi + \mu\psi \geq 0 \; \text{in}\; \cD\}.
\end{align*}
By $\cC^{2s+}(\cD)$ we denote the class of all continuous functions in $\Rd$ with the property that for any $f\in\cC^{2s+}(\cD)$ and any compact $\cK\subset\cD$ there 
exists $\alpha>0$ such that $f\in \cC^{2s+\alpha}(\cK)$. We also define
\begin{align*}
\widehat\sF^+(\cD, \mu) &=\{ \psi\in\cC(\Rd)\cap\cC^{2s+}(\cD)\cap L^1(\omega)\; :\; \psi\geq 0,\; \psi>0 \; \text{in}\; \cD, \; \bI \psi + \mu\psi \leq 0 \; \text{in}\; \cD\},
\\
\widehat\sF^-(\cD, \mu) &=\{ \psi\in\cC(\Rd)\cap\cC^{2s+}(\cD)\cap L^1(\omega)\; :\; \psi\leq 0,\; \psi<0 \; \text{in}\; \cD, \; \bI \psi + \mu\psi \geq 0 \; \text{in}\; \cD\}.
\end{align*}
Let
\begin{equation*}
\begin{aligned}
\Lambda^+=\Lambda^+(\cD)=\sup\{\mu \; :\; \sF^+(\cD, \mu)\neq\emptyset\}, & \quad  
\widehat\Lambda^+=\widehat\Lambda^+(\cD)=\sup\{\mu \; :\; \widehat\sF^+(\cD, \mu)\neq\emptyset\},
\\
\Lambda^-=\Lambda^-(\cD)=\sup\{\mu \; :\; \sF^-(\cD, \mu)\neq\emptyset\}, & \quad  
\widehat\Lambda^-=\widehat\Lambda^-(\cD)=\sup\{\mu \; :\; \widehat\sF^-(\cD, \mu)\neq\emptyset\}.
\end{aligned}
\end{equation*}
It is obvious that $0\leq \widehat\Lambda^+\leq \Lambda^+$ and $0\leq \widehat\Lambda^-\leq \Lambda^-$. Again from convexity it follows that
$\Lambda^+\leq \Lambda^-$ and $\widehat\Lambda^+\leq \widehat\Lambda^-$. In what follows, we shall use the notation $\Lambda^+, \Lambda^-$ 
(instead of $\Lambda^+(\cD), \Lambda^-(\cD)$) to indicate the principal eigenvalues. 

\begin{remark}
From the proof of Theorem~\ref{T2.4} it is easily seen that for $\lambda\geq 0$
\begin{align*}
\sF^+(\cD, \lambda) &=\{ \psi\in\cC(\Rd)\; :\; \psi\geq 0,\; \psi\gneq0 \; \text{in}\; \cD, \; \bI \psi + \lambda\psi \leq 0 \; \text{in}\; \cD\},
\\
\sF^-(\cD, \lambda) &=\{ \psi\in\cC(\Rd)\; :\; \psi\leq 0,\; \psi\lneq 0 \; \text{in}\; \cD, \; \bI \psi + \lambda\psi \geq 0 \; \text{in}\; \cD\},
\end{align*}
and therefore, we have
\begin{align*}
\Lambda^+ &=\sup\{\lambda\; :\;  \psi\in\cC(\Rd)\cap L^1(\omega)\;, \psi\geq 0,\; \psi\gneq 0 \; \text{in}\; \cD, \; \bI \psi + \lambda\psi \leq 0 \; \text{in}\; \cD\},
\\
\Lambda^- &=\sup\{\lambda\; :\; \psi\in\cC(\Rd)\cap L^1(\omega)\; ,\psi\leq 0,\; \psi\lneq 0 \; \text{in}\; \cD, \; \bI \psi + \lambda\psi \geq 0 \; \text{in}\; \cD\}.
\end{align*}
\end{remark}

\subsection{Statement of the main results}\label{S-main}
We now state the main results of this paper and the proofs of these results can be found in Section ~\ref{S-proofs}.
Our first result gives existence and uniqueness of eigenpairs. By $\cC_0(\cD)$ we denote the set of all continuous functions
in $\Rd$ that vanish on $\cD^c$.

\begin{theorem}\label{T1.1}
There exists functions $\Psi^+, \Psi^-\in\cC^{2s+}(\cD)\cap\cC_0(\cD)$ such that
\[\left\{ \begin{array}{llll}
\bI \Psi^+ = - \Lambda^+\Psi & \text{in}\; \cD,
\\[2mm]
\Psi^+ > 0 & \text{in}\; \cD,
\\[2mm]
\Psi^+=0 & \text{in}\; \cD^c,
\end{array}
\right.
\quad and \quad
\left\{ \begin{array}{llll}
\bI \Psi^- = - \Lambda^-\Psi & \text{in}\; \cD,
\\[2mm]
\Psi^- < 0 & \text{in}\; \cD,
\\[2mm]
\Psi^- = 0 & \text{in}\; \cD^c,
\end{array}
\right.
\]
where $\Lambda^{\pm}$ are positive. Moreover, $\Lambda^+$ (resp. $\Lambda^-$)
is the only eigenvalue corresponding to a positive (resp. negative) eigenfunction. Also both the eigenvalues
are geometrically simple.
\end{theorem}

Next result shows that $\Lambda^\pm$ are simple in  strong sense.

\begin{theorem}\label{T1.2}
Let $u \in\cC(\Rd)\cap L^1(\omega)$ be a viscosity solution of 
\begin{equation}\label{ET2.4A}
\begin{split}
\bI u + \Lambda^+ u &=0\quad \text{in}\; \cD,
\\
u &= 0 \quad \text{in}\; \cD^c,
\end{split}
\end{equation}
or of
\begin{equation}\label{ET2.4B}
\begin{split}
\bI u + \Lambda^+ u &\geq 0\quad \text{in}\; \cD,
\\[2mm]
u(x_0)>0, \quad u &\leq 0 \quad \text{in}\; \cD^c,
\end{split}
\end{equation}
for some $x_0\in\cD$. Then $u =t\Psi^+$ for some $t\in\RR$. A similar property holds for $(\Psi^-, \Lambda^-)$ when we reverse the inequalities above.
\end{theorem}

We could also show that the principal eigenvalues are isolated in the following sense.
\begin{theorem}\label{T1.3}
There exists $\varepsilon_0>0$, dependent on $d, \cD, \lambda, \Lambda$, such that
\begin{equation}\label{ET2.6A}
\begin{split}
\bI u + \mu u &=0\quad \text{in}\; \cD,
\\
u &= 0 \quad \text{in}\; \cD^c,
\end{split}
\end{equation}
has no solution $u\neq 0$ for $\mu\in (-\infty, \lambda^-+\varepsilon_0)\setminus\{\Lambda^+, \Lambda^-\}$.
\end{theorem}

Let us now state the maximum principles. We begin with the {\it refined maximum principles} for nonlocal operator. Recently in \cite{BL17b},
refined maximum principle is established for a general class of nonlocal Schr\"{o}dinger operators. The proof technique uses stochastic representation
of the eigenfunction which is not available for nonlinear problems.
\begin{theorem}\label{T1.4}
\
\begin{itemize}
\item[(a)] Let $u\in\cC(\Rd)\cap L^1(\omega)$ satisfy
\begin{equation*}
\begin{split}
\bI u + \mu u &\geq 0\quad \text{in}\; \cD,
\\
 u &\leq 0 \quad \text{in}\; \cD^c,
\end{split}
\end{equation*}
for some $\mu<\Lambda^+$, then it must hold that $u\leq 0$ in $\Rd$.
\item[(b)] Let $u\in\cC(\Rd)\cap L^1(\omega)$ satisfy
\begin{equation*}
\begin{split}
\bI u + \mu u &\leq 0\quad \text{in}\; \cD,
\\
u &\geq 0 \quad \text{in}\; \cD^c,
\end{split}
\end{equation*}
for some $\mu<\Lambda^-$, then it must hold that $u\geq 0$ in $\Rd$.
\end{itemize}
\end{theorem}
As a corollary to Theorem~\ref{T1.4} we obtain the following comparison result.
\begin{corollary}\label{C2.1}
Suppose that $\mu<\Lambda^+$ and $f\in\cC(\bar\cD)$. Suppose that for some $u, v\in\cC(\Rd)\cap L^1(\omega)$ we have
\begin{align*}
\bI u + \mu u\geq f\geq \bI v + \mu v\quad & \text{in}\; \cD,
\\
u\leq v\quad & \text{in}\; \cD^c.
\end{align*}
Then it holds that $u\leq v$ in $\Rd$.
\end{corollary}

\begin{proof}
Note that $\bI$ is the extremal operator with respect to the class $\cL$ in the sense of \cite[p.~602]{CS09}. Therefore, by \cite[Theorem~5.9]{CS09}
we obtain that 
\begin{equation*}
\begin{split}
\bI (u-v) + \mu (u-v) &\geq 0\quad \text{in}\; \cD,
\\
(u-v) &\leq 0 \quad \text{in}\; \cD^c.
\end{split}
\end{equation*}
Hence $u\leq v$ in $\Rd$ follows from Theorem~\ref{T1.4}(a).
\end{proof}

Let $\lambda(L, \cD)$ be the principal eigenvalue of the operator $L$ with Dirichlet exterior condition.

\begin{corollary}
It follows that 
$$\Lambda^+\leq \inf_{L\in\cL}\, \lambda(L, \cD)\leq \sup_{L\in\cL}\, \lambda(L, \cD)\leq \Lambda^-.$$
\end{corollary}

\begin{proof}
Let $\Psi_L$ be the positive eigenfunction of the operator $L$ corresponding to $\lambda(L, \cD)$. It then follows that
$$0= L\Psi_L + \lambda(L, \cD)\Psi_L\leq \bI\Psi_L + \lambda(L, \cD)\Psi_L, \quad \text{and}\quad \Psi_L>0\quad \text{in}\; \cD.$$
Therefore, by Corollary~\ref{C2.1}(a) we must have $\Lambda^+\leq \lambda(L, \cD)$ and hence, $\Lambda^+\leq \inf_{L\in\cL}\, \lambda(L, \cD)$.
Also note that 
$$0= L(-\Psi_L) + \lambda(L, \cD)(-\Psi_L)\leq \bI(-\Psi_L) + \lambda(L, \cD)(-\Psi_L).$$
Then by definition of $\Lambda^-$ we get $\Lambda^-\geq \lambda(L, \cD)$ implying $\Lambda^-\geq \sup_{L\in\cL}\lambda(L, \cD)$.
\end{proof}

We also have the following.

\begin{theorem}\label{T1.5}
Suppose that for some $u\in\cC(\Rd)\cap L^1(\omega)$ we have
\[
\left\{\begin{array}{lll}
\bI u + \mu u \leq 0, & \text{in}\; \cD,
\\[2mm]
u\geq 0\; \text{in}\; \Rd,\; \; u>0 & \text{in}\; \cD,
\end{array}
\right.
\quad resp,\quad 
\left\{
\begin{array}{lll}
\bI u + \mu u\geq 0, & \text{in}\; \cD,
\\[2mm]
u\leq 0\; \text{in}\; \Rd,\; \; u<0 & \text{in}\; \cD.
\end{array}
\right.
\]
Then either $\mu<\Lambda^+$ or $\mu=\Lambda^+$ with $u=t\Psi^+$ for some $t>0$ (resp.
either $\mu<\Lambda^-$ or $\mu=\Lambda^-$ with $u=t\Psi^-$ for some $t>0$).
\end{theorem}

An interesting corollary to the above result is the following strict monotonicity property of the principal eigenvalues.
\begin{corollary}\label{C2.2}
Suppose that $\cD_1\subsetneq\cD_2$. Then $\Lambda^+(\cD_1)> \Lambda^+(\cD_2)$ and $\Lambda^-(\cD_1)> \Lambda^-(\cD_2)$.
\end{corollary}

Next we prove continuity of the the principal eigenvalues with respect to the domains.
\begin{theorem}\label{T1.6}
Let $\cD_n$ be a collection of smooth domains converging to a smooth domain $\cD$. Then we have $\Lambda^+(\cD_n)\to \Lambda^+(\cD)$ and $\Lambda^-(\cD_n)\to \Lambda^-(\cD)$ as $n\to\infty$.
\end{theorem}

\begin{remark}
From Corollary~\ref{C2.2} and Theorem~\ref{T1.6} it is easy to see that
\begin{align*}
\Lambda^+(\cD) &=\sup\{\lambda\; :\;  \psi\in\cC(\Rd)\cap L^1(\omega)\;, \psi\geq 0,\; \inf_{\cD}\psi \geq 1 , \; \bI \psi + \lambda\psi \leq 0 \; \text{in}\; \cD\},
\\
\Lambda^-(\cD) &=\sup\{\lambda\; :\; \psi\in\cC(\Rd)\cap L^1(\omega)\; ,\psi\leq 0,\; \sup_{\cD}\psi\leq -1 , \; \bI \psi + \lambda\psi \geq 0 \; \text{in}\; \cD\}.
\end{align*}
\end{remark}

Similar to the local case we also have an existence result.
\begin{theorem}\label{T1.7}
Let $\cD$ be a smooth domain and $f\in\cC(\bar\cD)$. Then for every $\mu<\Lambda^+$ there exists a unique solution $u\in\cC(\Rd)$ to the equation
\begin{equation}\label{ET2.10A}
\begin{split}
\bI u + \mu u  &=f\quad \text{in}\; \cD,
\\
u &= 0 \quad \text{in}\; \cD^c.
\end{split}
\end{equation}
\end{theorem}

We also prove an anti-maximum principle.
\begin{theorem}\label{T1.8}
Let $f\in\cC(\bar\cD)$  and $f\lneq 0$.  There exists $\eta>0$, dependent on $f, \cD, \lambda, \Lambda,$ such that for
any solution $u$ of
\begin{equation}\label{ET2.7A}
\begin{split}
\bI u + \mu u &=f\quad \text{in}\; \cD,
\\
u &= 0 \quad \text{in}\; \cD^c,
\end{split}
\end{equation}
with $\mu\in (\Lambda^-, \Lambda^-+\eta)$ we must have $u<0$ in $\cD$.
\end{theorem}

Anti-maximum principle was first proved in \cite{CP} for linear elliptic PDE. Now there are several extensions of anti-maximum principal
available for elliptic PDE. Let us also mention \cite{A08} where anti-maximum principle was obtained for fully nonlinear elliptic operators.
Very recently in \cite{BL17b}, a weak anti-maximum principle has been established for linear nonlocal schr\"{o}dinger operator.

Finally, we would  address the existence of non-trivial solution of the equation
\begin{equation}\label{E1.10}
\sM^+_* u = -\mu u + f (\mu, u) \quad \text{in}\; \cD, \quad\text{and}\quad u=0\quad \text{in}\; \cD^c.
\end{equation}
We assume that $f$ is continuous, $f(\mu, t)=\sorder(\abs{t})$ as $t\to 0$ uniformly in $\mu\in\RR$. Note that by definition $(\mu, 0)$ is a solution to \eqref{E1.10} for any $\mu\in\RR$.
These solutions are known as trivial solutions. Let $\cS$ be the collection of all non-trivial solution to \eqref{E1.10}. Recall that a pair
$(\mu, 0)$ is said to be a \textit{bifurcation point} if every neighbourhood of $(\mu, 0)$ in $\RR\times\cC_0(\cD)$ intersects $\cS$.

We establish the following bifurcation result (compare with \cite[Theorem~1.3]{BEQ})
\begin{theorem}\label{T1.9}
The pairs $(\Lambda^+, 0)$ and $(\Lambda^-, 0)$ are bifurcation points of the solutions to \eqref{E1.10}. Moreover, if $\Lambda^+<\Lambda^-$ then $(\Lambda^+, 0)$ (resp., $(\Lambda^-, 0)$)
is a bifurcation point of positive (resp., negative) solutions. Furthermore, the connected component of $\cS$ containing $(\Lambda^+, 0)$ (resp., $(\Lambda^-, 0)$) in its closure is either unbounded or contains a pair
$(\hat\mu, 0)$ for some $\hat\mu\neq\Lambda^+$ (resp., $\hat\mu\neq \Lambda^-$).
\end{theorem}

Before we conclude this section let us remark that though we have developed the results for certain nonlinear operators, all the results continue 
to hold if we assume $\bI$ is convex and elliptic with respect to $\cL_*$. The only omittance from the results would be $\cC^{2s+}$ property (Theorem~\ref{S-Thm1})
and rest of the statement in the theorems stay as it is.

\subsection{Motivation for our problem}\label{S-moti}
 Sup-type (or inf-type) operators are very common in control theory. In particular, $\bI$ in \eqref{Opt} appears
in controlled eigenvalue problems \cite{BB10}. It is generally expected that, at least when $\cL$ is sufficiently large,
$$\Lambda^+=\inf_{L\in\cL} \lambda(L, \cD), \quad \text{and}\quad \Lambda^-=\sup_{L\in\cL}\, \lambda(L, \cD),$$
where $\lambda(L, \cD)$ denotes the principal (Dirichlet) eigenvalue of the linear operator $L$ in $\cD$. When $L$ is given by a (local) non-degenerate elliptic
operator the above relations are well known \cite{BB10, BEQ}. Eigenvalue problems appear in various places of control theory-- for example, in risk-sensitive
control, reliability theory, large deviation theory etc. Let $(\tilde\cD, \tilde\sF, \Prob)$ be a probability space and $X^L$ be the L\'evy process
defined on $(\tilde\cD, \tilde\sF, \Prob)$ with
generator being $L$. It is known that \cite[Corollary~4.1]{BL17b}
$$\lambda(L, \cD)=-\lim_{T\to\infty} \, \frac{1}{T}\log \Prob(\uptau^L>T),$$
where $\uptau^L$ denotes the exit time of the process $X^L$ from the domain $\cD$. Thus, we have the following relations from above
$$\Lambda^+=\inf_{L\in\cL}\left[ -\lim_{T\to\infty} \, \frac{1}{T}\log \Prob(\uptau^L>T)\right], \quad \text{and}\; 
\Lambda^-=\sup_{L\in\cL}\, \left[ -\lim_{T\to\infty} \, \frac{1}{T}\log \Prob(\uptau^L>T)\right].$$
Therefore, $\Lambda^\pm$ represent the optimal value of the \textit{long time exit rates} of the processes associated to the set $\cL$.

\section{Preliminaries}\label{S-prelim}

In this section we review some results that will be essential in proof of our main results.
We start by stating the comparison principle from \cite{CS09} (see also \cite[Corollary~2.9]{CL12}). Let us remark that the
comparison result in \cite{CS09} is obtained for bounded solutions and can easily be extended to the class $L^1(\omega)$. This 
can be done by truncating the functions and using stability of viscosity solution. See \cite[Corollary~2.9]{CL12} for more details.

\begin{theorem}[Comparison theorem]\label{CS-Thm1}
Let $\bI$ be any operator satisfying ellipticity condition with respect to $\cL_0$ and $u, v$ be two functions in $L^1(\omega)$. Suppose that
\begin{itemize}
\item[(i)] $u$ is upper semicontinuous and $v$ is lower semicontinuous in $\bar\Omega$. 
\item[(ii)] $\bI u\geq f$ and $\bI v\leq f$ in $\Omega$, and $f$ is continuous in $\Omega$.
\item[(iii)] $u\leq v$ in $\Omega^c$.
\end{itemize}
Then we have $u\leq v$ in $\Rd$.
\end{theorem}
Using the comparison theorem above and Perron's method one can establish the existence result \cite{CS09, M17}.

\begin{theorem}\label{CS-Thm2}
Let $f\in\cC(\bar\Omega)$. Then there exists a unique viscosity solution $u\in\cC(\Rd)$ satisfying 
$$\bI u=f \quad \mbox{in}\; \Omega, \quad \text{and}\quad u=0\quad  \text{in}\; \Omega^c.$$
\end{theorem}

The following result gives the boundary behaviour of the solution \cite[Theorem~1.5]{RS17}.

\begin{theorem}\label{RS-Thm1}
Let $\delta(\cdot)$ be continuous non-negative function that coincide with the $\dist(x, \cD^c)$  in a neighbourhood of $\partial\Omega$. Then for
$$\bI u=f \quad \mbox{in}\; \Omega, \quad \text{and}\quad u=0,\; \text{in}\; \Omega^c,$$
there exists $\alpha$, dependent on $d, s, \lambda, \Lambda$, such that
$$\norm{\frac{u}{\delta^s}}_{\cC^{\alpha}(\bar\Omega)}\leq C \norm{f}_{L^\infty(\bar\Omega)},$$
where $C$ is dependent on $\Omega, s, \gamma, \lambda, \Lambda$.
\end{theorem}
The above boundary regularity is not sharp. Under additional regularity assumptions on $k$, a sharp boundary behaviour is established in \cite[Theorem~1.3]{RS16}.
We also recall the following barrier functions from \cite{RS16}. Recall that $\sM^{\pm}_*$ denote the extremal operators with respect to $\cL_*$.
By $B_r$ we denote the ball of radius $r$ in $\Rd$ centered at $0$.

\begin{lemma}[Lemma~3.3 of \cite{RS16}]\label{RS-Lem1}
There exist positive constants $\varepsilon, C$ and a radial bounded, continuous function $\varphi_1$ which is $\cC^{1, 1}$ in $B_{1+\varepsilon}\setminus \bar B_1$ satisfying
\[
\left\{ \begin{array}{llll}
\sM^+_*\varphi_1  \leq -1 & \text{in}\; B_{1+\varepsilon}\setminus \bar B_1,
\\
\varphi_1(x) =  0 & \text{in}\; B_1,
\\
0\leq \varphi_1 \leq C(\abs{x}-1)^s & \text{in}\; B^c_1,
\\
\varphi_1 \geq 1 & \text{in}\; B^c_{1+\varepsilon},
\end{array}
\right.
\]
where the constants $\varepsilon, C$ depends only on $d, s, \lambda, \Lambda$.
\end{lemma}

\begin{lemma}[Lemma~3.4 of \cite{RS16}]\label{RS-Lem2}
There exists positive constants $c$ and a radial bounded, continuous function $\varphi_2$ satisfying
\[
\left\{ \begin{array}{llll}
\sM^-_*\varphi_2  \geq c & \text{in}\; B_{1}\setminus \bar B_{\frac{1}{2}},
\\
\varphi_2(x) =  0 & \text{in}\; B^c_1,
\\
\varphi_2 \geq c(1-\abs{x})^s & \text{in}\; B_1,
\\
\varphi_2 \leq 1 & \text{in}\; \bar B_{\frac{1}{2}},
\end{array}
\right.
\]
where the constants $c$ is dependent only on $d, s, \lambda, \Lambda$.
\end{lemma}

Let us now prove the Hopf's lemma using the barrier functions. Let $\rin$ be a radius of interior sphere condition in $\cD$.

\begin{theorem}[Hopf's Lemma]\label{T2.4}
Let $\sM^-_*u\leq 0$ in $\Omega$ where $u\in\cC(\Rd)$ is non-negative and  vanishes in $\Omega^c$. Then either $u=0$ in $\Rd$ or $u>0$ in $\cD$. Moreover,
if $u$ is non-zero then there exists $c_1>0$ such that for any $x_0\in\partial\Omega$ we have
$$\frac{u(x)}{(\rin-\abs{x-z})^s}\geq c_1, \quad \text{for all}\; x\in B_{\rin}(z),$$
where $B_{\rin}(z)\subset\Omega$ is a ball of radius $\rin$ around $z$ that touches $\partial\Omega$ at $x_0$. 
\end{theorem}

\begin{proof}
Since $\cL_*\subset\cL_0$ it is easy to see that $\sM^-_0 u\leq 0$ in $\Omega$. Define $K=\{x\in\cD \; :\; u(x)=0\}$. We claim that either $K=\emptyset$ or $K=\Omega$. Suppose
that $K\neq\emptyset$. Note that $K$ is a closed subset of $\Omega$. We claim that $K$ is also open in $\Omega$. Then the claim follows from the fact that $\Omega$ is a connected set.
Suppose that $z\in K$ and $r>0$ be such that $B_{2r}(z)\subset \Omega$. Then by \cite[Theorem~10.4]{CS09} we have $\varepsilon>0$ and a constant $\kappa$, dependent on $\varepsilon, \lambda, \Lambda, d, s$
satisfying
$$\abs{\{u> t\}\cap B_r(z)} \leq \kappa\, r^d\, u(z) t^{-\varepsilon}=0, \quad \text{for all}\; t>0.$$
Letting $t\to 0$, this of course, implies that $\abs{\{u> 0\}\cap B_r(z)}=0$ and therefore, by continuity of $u$ we have $B_r(z)\subset K$. Thus $K$ is open. Hence we have the claim.

Now we prove the second part of the theorem. Since $u$ is non-zero we must have that $u>0$ in $\Omega$. For $\hat{r}=\frac{\rin}{8}$ we define
$$\Omega_{\hat r}=\{x\in\Omega \; :\; \dist(x, \partial\Omega)\geq \hat{r}\}.$$
Let $\hat{m}=\min_{\Omega_{\hat r}} u$. Let $z\in\Omega$ be such that $B_{\rin}(z)\subset\Omega$ touches $x_0$. In fact, $z$ would lie on the inward normal vector at $x_0$ to $\partial\Omega$.
Consider
$$\psi(y) = \hat{m}\, \varphi_2(\frac{y-z}{\rin}), $$
where $\varphi_2$ is the barrier function in Lemma~\ref{RS-Lem2}. Then it is easy to see from Lemma~\ref{RS-Lem2} that
\[
\left\{ \begin{array}{llll}
\sM^-_*\psi  \geq 0 & \text{in}\; B_{\rin}\setminus \bar B_{\frac{\rin}{2}},
\\
\psi(y) =  0 & \text{in}\; B^c_{\rin}(z),
\\
\psi(y) \geq \frac{\hat{m} c}{(\rin)^s}(\rin-\abs{y-z})^s & \text{in}\; B_{\rin}(z),
\\
\psi \leq m & \text{in}\; \bar B_{\frac{\rin}{2}}(z).
\end{array}
\right.
\]
Hence applying Theorem~\ref{CS-Thm1} we find that
$$u(y)\geq \psi(y) \geq\; \frac{\hat{m} c}{(\rin)^s}(\rin-\abs{y-z})^s, \quad \text{in}\; B_{\rin}(z).$$
This complete the proof.
\end{proof}

We also need the following regularity estimate  from \cite[Theorem~1.3]{S15} (see also \cite{K13}). It should be noted that the results of \cite{S15} are obtained for a concave operator (inf-type) and therefore, analogous results hold for sup-type operators as well.

\begin{theorem}\label{S-Thm1}
Let $\bI$ be the operator given by \eqref{Opt}.
 Consider the equation $\bI u= f$ in $B_1$. Then there exists $\bar\alpha>0$ , dependent on $d, s, \lambda, \Lambda$, such that
for any $\alpha\in (0, \bar\alpha)$ with $2s+\alpha$ not being an integer, we have
$$\norm{u}_{\cC^{2s+\alpha}(B_{\frac{1}{2}})} \leq C (\norm{u}_{\cC^\alpha(B^c_1)} + \norm{f}_{\cC^\alpha(B_1)}),$$
where $C$ depends only on $d, \alpha, s, \lambda, \Lambda$.
\end{theorem}

The following result gives regularity of the solutions upto the boundary. When $\bI$ is given by the fractional Laplacian operator, analogous result is proved in \cite{RS14}.
\begin{theorem}\label{T2.6}
Let $u$  be a viscosity solution of $\bI u = f$ in $\Omega$ and $u=0$ in $\Omega^c$. Then for some constant $\alpha, C$, dependent on $d, s, \Omega, \lambda, 
\Lambda,$ we have
\begin{equation}\label{ET1.60}
\norm{u}_{\cC^\alpha(\bar\Omega)}\leq C \norm{f}_{L^\infty(\cD)}.
\end{equation}
Furthermore, if $s\in (\frac{1}{2}, 1)$ then we get
\begin{equation}\label{ET1.601}
\norm{u}_{\cC^s(\bar\Omega)}\leq C \norm{f}_{L^\infty(\cD)}.
\end{equation}
\end{theorem}

\begin{proof}
It is well known that $\norm{u}_{L^\infty(\cD)}\leq C \norm{f}_{L^\infty(\cD)}$. Let $x_0\in\partial\cD$ be any point and $\rout\in(0, 1)$ be a radius of 
exterior sphere. Let $z\in\cD^c$ be such that $B_{\rout}(z)$ touches $\cD$ at the point $x_0$. Recall $\varphi_1$ from Lemma~\ref{RS-Lem1}. Define
$$\psi(y) = (\norm{u}_{L^\infty(\cD)}+\norm{f}_{L^\infty(\cD)}) \, \varphi_1\left(\frac{y-z}{\rout}\right).$$
Then it is easy to check that
\[
\left\{ \begin{array}{llll}
\sM^+_*\psi  \leq  f & \text{in}\; B_{(1+\varepsilon)\rout}(z)\setminus \bar B_{\rout}(z),
\\
\psi(y) =  0 & \text{in}\; B_{\rout}(z),
\\
0\leq \psi(y) \leq C(\frac{|y-z|}{\rout}-1)^s & \text{in}\; B^c_{\rout}(z),
\\
\psi \geq \norm{u}_\infty & \text{in}\; B^c_{(1+\varepsilon)\rout}(z).
\end{array}
\right.
\]
Therefore, by comparison principle (Theorem~\ref{CS-Thm1}) we must have 
$$u(y)\leq C \norm{f}_{L^\infty(\cD)} \left(\frac{|y-z|}{\rout}-1\right)^s, \quad y\in B^c_{\rout}(z).$$
Since $x_0$ is an arbitrary point, we note that in a neighbourhood $N(\partial\cD)$ of $\partial\cD$ we have
$$ u(y)\leq C \norm{f}_{L^\infty(\cD)} \delta^s(y), \quad y\in N(\partial\cD).$$
Note that $\sM^+_*(-u)\geq -f$, and therefore, again using the barrier function we can obtain a similar boundary behaviour for $-u$. This of course,
implies
\begin{equation}\label{ET1.6A}
\norm{\frac{u}{\delta^s}}_{L^\infty(\cD)}\leq C \norm{f}_{L^\infty(\cD)}.
\end{equation}
The rest of the argument follows from a usual scaling method (see \cite{ABC}). Denote by $\omega(y)=(1+\abs{y})^{-d-2s}$.
From \cite[Theorem~2.6]{CS11a} we know that if $-C_0\leq \bI u\leq C_0$
in $B_1$, then there exists $\alpha_1>0$ such that
\begin{equation}\label{ET1.6B}
\norm{u}_{\cC^{\alpha_1}(B_{\frac{1}{2}})}\leq\; C\left(\norm{u}_{L^\infty(B_1)} + \norm{u}_{L^1(\Rd, \omega)} + C_0\right).
\end{equation}
Therefore, in order to obtain \eqref{ET1.60} we need to consider the behaviour near the boundary. More precisely, we show that $u$ is in 
$\cC^\alpha, \alpha=\min\{\alpha_1, s\}$, in
a neighbourhood of the boundary.  Let $x\in \cD$ with $\dist(x, \cD)=2r$. Without loss of generality we assume that $x=0$, otherwise we
translate $u$. Define $v(y)=u(ry)$ in $\Rd$. Due to \eqref{ET1.6A} we note that for some constant $C$, independent of $u$ and $f$, we have
\begin{equation}\label{ET1.6C}
\abs{v(y)}\leq C (\norm{u}_{L^\infty(\cD)} + \norm{f}_{L^\infty(\cD)}) r^s (1+\abs{y})^s, \quad y\in\Rd.
\end{equation}
We also note that
$$-r^{2s} \norm{f}_{L^\infty(\cD)} \;\leq\; \bI v\leq\; r^{2s} \norm{f}_{L^\infty(\cD)}, \quad \text{in}\; B_1.$$
Hence by \eqref{ET1.6B} we have 
$$\norm{v}_{\cC^{\alpha}(B_{\frac{1}{2}})}\leq\; C\left(\norm{v}_{L^\infty(B_1)} + \norm{v}_{L^1(\Rd, \omega)} + r^{2s}\norm{f}_{L^\infty(\cD)}\right).$$
From \eqref{ET1.6C} we observe that 
$$r^{-s}\left(\norm{v}_{L^\infty(B_1)}+\norm{v}_{L^1(\Rd, \omega)}\right) \;\leq\; C \norm{f}_{L^\infty(\cD)}.$$
Thus we have for some constant $C$ that
$$\norm{v}_{\cC^{\alpha_1}(B_{\frac{1}{2}})}\leq\; C r^s \norm{f}_{L^\infty(\cD)} .$$
This of course, implies that
\begin{equation}\label{ET1.6D}
\norm{u}_{\cC^{\alpha}(B_{\frac{r}{2}})} \leq C \norm{f}_{L^\infty(\cD)}.
\end{equation}
Let $x_1, x_1\in \cD$ and $r_i=\dist(x_i, \partial\cD)$ for $i=1,2$.
If $\abs{x_1-x_2}\leq 4^{-1}r_i$ for some $i$, then we consider the ball $B_{\frac{r_i}{2}}(x_i)$, and use \eqref{ET1.6D} to obatin 
$\abs{u(x_1)-u(x_2)}\leq C \abs{x_1-x_2}^{\alpha_1} \norm{f}_\infty$. Otherwise, we must have $\abs{x_1-x_2}\geq 4^{-1}(r_1\vee r_2)$. In this case,
using \eqref{ET1.6A} we get that
\begin{align*}
\frac{\abs{u(x_1)-u(x_2)}}{\abs{x_1-x_2}^s} \leq 4 \left(\frac{\abs{u(x_1)}}{r^s_1}+\frac{\abs{u(x_2)}}{r^s_2}\right)\leq 4 C\, \norm{f}_{L^\infty(\cD)},
\end{align*}
by \eqref{ET1.6A}. This completes the proof of \eqref{ET1.60}.

Now suppose $s>1/2$. Then \eqref{ET1.601} follows employing the same technique as above and using \cite[Theorem~4.1]{K13} which state that
if $-C_0\leq \bI u\leq C_0$
in $B_1$, then there exists $\alpha_1>0$ such that
\begin{equation*}
\norm{u}_{\cC^{1, \alpha_1}(B_{\frac{1}{2}})}\leq\; C\left(\norm{u}_{L^\infty(B_1)} + \norm{u}_{L^1(\Rd, \omega)} + C_0\right).
\end{equation*}
\end{proof}

\section{Proofs of main results}\label{S-proofs}

\subsection{Nonlinear Krein-Rutman theorem}
To establish the existence of principal eigenvalues we use nonlinear Krein-Rutman theorem. To state the theorem in its abstract setting we need few 
definitions which we introduce below. Let $(\cX, \norm{\cdot})$ be a Banach space and $\sP\subset\cX$ be a non-trivial closed convex cone
 with the property that $\sP\cap(-\sP)=\{0\}$. Denote by $\dot\sP=\sP\setminus\{0\}$.
 
We write $x\preceq y$ if $y-x\in\sP$, and $x\prec y$ if $x\preceq y$ and $x\neq y$. Let $\cT:\cX\to\cX$ be a function. $\cT$ is said to be {\it increasing}
if $x\preceq y\Rightarrow \cT x\preceq \cT y$, and it is said to be {\it strictly increasing} if $x\prec y\Rightarrow \cT x\prec \cT y$, and it is said to be
{\it strongly increasing} if  $x\prec y\Rightarrow  \cT y-\cT x\in\mathrm{int}(\sP)$, provided $\mathrm{int}(\sP)$ is nonempty. $\cT$ is called positively
$1$-homogeneous if $\cT(tx)=t\cT x$ for all $t>0$ and $x\in\cX$. $T$ is called {\it strongly positive} if $\cT(\dot\sP)\subset\mathrm{int}(\sP)$.
Also, a map $\cT :\cX\to\cX$ is called {\it completely continuous} if it is continuous and compact.
We also say $\cT$ is {\it superadditive} if for any $x, y\in\cX$ we have $\cT x+ \cT y\preceq \cT(x+y)$. Note that if $\cT 0=0$ then strongly increasing
implies strongly positive.

The nonlinear Krein-Rutman theorem states the following (see \cite[Theorem~4.1]{ABS}, \cite{A18, M07}).
\begin{theorem}\label{NKR}
Let $\cT : \cX\to\cX$ be an increasing, positively $1$-homogeneous, completely continuous map such that for some
$x_0\in\sP$ and $M > 0$, $x_0\preceq M \cT x_0$. Then there exists $\hat\lambda\in\RR$ and $\hat{x}\in\dot\sP$
such that $\cT\hat{x} = \hat\lambda \hat{x}$. Moreover , if $\mathrm{int}(\sP)\neq\emptyset$ and $\cT$ is strongly increasing
then $\hat\lambda>0$ is the unique eigenvalue with an eigenvector in $\dot\sP$. Furthermore, if $\cT$ is  superadditive
then $\hat{x}$ is the unique eigenvector in $\dot\sP$ with $\norm{\hat{x}}=1$.
\end{theorem}

\begin{remark}
The requirement of superadditivity in Theorem~\ref{NKR} can be replaced by the following \cite{A18}:  for any $x, y\in\sP$ we have $\cT(y-x)\preceq \cT y-\cT x$.
To see this, suppose $\hat{x}, \hat{y}\in\mathrm{int}(\sP)$ be two eigenfunctions corresponding to the eigenvalue $\hat\lambda$. Choose $\alpha>0$ such that 
$\hat{x}-\alpha\hat{y}\in\partial\sP\setminus{0}$. Then 
$$\cT(\hat{x}-\alpha\hat{y})\preceq \cT\hat{x}-\alpha\cT{\hat y} =\hat{\lambda}(\hat{x}-\alpha\hat{y}).$$
But this is a contradiction since $\cT$ is strongly increasing implying $\cT(\hat{x}-\alpha\hat{y})\in \mathrm{int}(\sP)$ .
\end{remark}

We shall use the above theorem to find the principal eigenvalues of the operator $\bI$. Fix an $\varepsilon\in (0, s\wedge(1-s)\wedge\alpha)$ where $\alpha$ is chosen from Theorem~\ref{RS-Thm1} and ~\ref{T2.6}.
Recall that $\cC_0(\cD)$ denote  the set of all continuous functions in $\Rd$ that vanish in $\cD^c$. Define 
$$\cX=\{f\in \cC_0(\cD)\cap\cC^\varepsilon(\bar\cD)\; :\: \sup_{\cD}\abs{f/\delta^{s}}<\infty\}.$$
The norm on $\cX$ is given by
$$\norm{f}_\cX= [f]_{\cC^\varepsilon(\bar\cD)} + \sup_{\cD}\abs{f/\delta^{s}},$$
where $[\cdot]_{\cC^\varepsilon(\bar\cD)}$ denotes the H\"{o}lder seminorm. It is easily seen that $(\cX, \norm{\cdot}_\cX)$ is a Banach space.
We take $\sP=\{ f\in\cX\: :\; f\geq 0\}$. Let us now define $\cT$. For any $f\in\cX$, we know by Theorem~\ref{CS-Thm2} that there exists unique
$u$ satisfying $\bI u= -f$ in $\cD$ with $u=0$ in $\cD^c$. From Theorem~\ref{RS-Thm1} it is easily seen that $u\in\cX$. We define
$\cT f =u$. Therefore, we have
\begin{equation}\label{E2.1}
\bI (\cT f) = -f, \quad \text{in}\; \cD, \quad \text{for}\; f\in\cX.
\end{equation}
Also define $\tilde\cT(f)=-\cT(-f)$ for $f\in\cX$.
We verify that $\cT, \tilde\cT$ have all the properties required to apply Theorem~\ref{NKR} above.

\begin{lemma}\label{L2.1}
The maps $\cT, \tilde\cT$ defined above are positively $1$-homogeneous, increasing, completely continuous, strongly increasing. Moreover, $\tilde\cT$ is superadditive.
\end{lemma}

\begin{proof}
From the structure of $\bI$ and the uniqueness result (Theorem~\ref{CS-Thm1}) it is obvious that $\cT$ is positively $1$-homogeneous.
Again comparison principle shows that $\cT$ is increasing. So
we prove the remaining claims below.

\underline{Completely continuous:}\, Let $\{\varphi_n\}$ be a sequence in $\cX$ with the property that $\sup_n\norm{\varphi_n}_\cX<\infty$.
Let $\cT \varphi_n=u_n$. By Theorem~\ref{RS-Thm1} and \ref{T2.6} we get that
\begin{equation}\label{EL2.1A}
\sup_{n} \left(\norm{u_n \delta^{-s}}_{\cC^{\alpha}(\bar\cD)}+ \norm{u_n}_{\cC^\alpha(\bar\cD)}\right) < \infty.
\end{equation}
Therefore, using \eqref{EL2.1A}, we can extract a subsequence $\{\varphi_{n_k}\}$ and $\{u_{n_k}\}$ such that 
$$\norm{\varphi_{n_k}-\varphi}_{\cC(\bar\cD)} + \norm{u_{n_k}-u}_\cX\to 0, \quad \text{as}\; n_k\to\infty,$$
for some $\varphi, u\in\cX$. By stability property of the viscosity solutions \cite[Corollary~4.7]{CS09} it then follows that 
$$\bI u=\varphi, \quad \text{in}\; \cD, \quad \text{and}\; \; u=0\; \; \text{in}\; \cD^c.$$
Due to uniqueness it must hold that $\cT\varphi=u$. This gives compactness. Continuity of $\cT$ follows from a similar argument.

\underline{Strongly increasing:}\, Let $\varphi_1\prec \varphi_2$ with $\varphi_i\in\cX$ for $i=1, 2$. Since
\begin{equation}\label{EL2.1B}
\bI (\cT \varphi_2) = -\varphi_2\leq -\varphi_1=\bI(\cT \varphi_1),
\end{equation}
using Theorem~\ref{CS-Thm1} we find that $\cT\varphi_1\leq \cT\varphi_2$ in $\Rd$. Let $u_i=\cT\varphi_i$ for $i=1,2$, and $u=u_2-u_1$.
Since $u_i$ is a classical solution, by Theorem~\ref{S-Thm1}, we obtain from \eqref{EL2.1B} that
\begin{equation}\label{EL2.1C}
\sM^{-}_* u \leq 0, \quad \text{in}\; \cD, \quad \text{and}\quad u\geq 0\quad \text{in}\; \Rd.
\end{equation}
By Theorem~\ref{T2.4} and \eqref{EL2.1C} we either have $u=0$ in $\Rd$ or $u>0$ in $\cD$.
 Since $\varphi_1\neq \varphi_2$, $u=0$ in $\Rd$ can not hold. Thus we must
have $u>0$ in $\Rd$. Again employing Theorem~\ref{T2.4} we note that
\begin{equation}\label{EL2.1D}
\min_{\bar\cD}\tfrac{u}{\delta^s}\geq \kappa,
\end{equation}
for some $\kappa$ positive. Now for $\kappa_1\in (0, \kappa/2)$ 
consider the open set $N_{\kappa_1}=\{\psi\in\cX\; :\; \norm{u-\psi}_\cX<\kappa_1\}$ in $\cX$. We claim that $N_{\kappa_1}\subset\sP$.
Indeed, for any $x\in\cD$, and $\psi\in N_{\kappa_1}$
$$\tfrac{\psi(x)}{\delta^s(x)}\geq \tfrac{u(x)}{\delta^s(x)}-\kappa_1\geq \kappa-\kappa_1\geq \kappa/2>0,$$
by \eqref{EL2.1D}. Thus $\psi\in\sP$ implying $N_{\kappa_1}\subset\sP$. Hence $u\in\mathrm{int}(\sP)$.

\underline{Superadditivity of $\tilde\cT$:}\, Consider $\varphi_i\in\cX$ for $i=1,2,$ and denote $u_i=\tilde\cT \varphi_i$. By Theorem~\ref{S-Thm1} we note that $u_i$ is a classical solution to
$\bI (-u_i)=\varphi_i$ for $i=1,2$. Let $u=\tilde\cT(\varphi_1+\varphi_2)$.
Since $\bI$ is convex and positively $1$-homogeneous we have
$$\bI (-u) = \varphi_1+\varphi_2=\bI (-u_1) +\bI (-u_2)\geq \bI(-u_1-u_2).$$
Thus by Theorem~\ref{CS-Thm1} we obtain $u\geq u_1+u_2$ implying $\cT\varphi_1+\cT\varphi_2\preceq \cT(\varphi_1+\varphi_2)$.
\end{proof}

\begin{lemma}\label{L2.2}
Let $\varphi\in\dot\sP$. Then there exists $M>0$ such that $\varphi\preceq M \cT \varphi$. Similar property holds for $\tilde\cT$.
\end{lemma}

\begin{proof}
Suppose that the above claim is not true. Then for every $M>0$ it must hold that $\cT\varphi-\frac{1}{M}\varphi\notin\sP$. Letting $M\to\infty$ it 
follows that $\cT\varphi\notin\mathrm{int}(\sP)$. But this contradicts strongly increasing property of $\cT$ established in Lemma~\ref{L2.1}.
\end{proof}

\begin{lemma}\label{L2.3}
Let $\hat\lambda>0$ be an eigenvalue of $\cT$ with an eigenvector in $\sP$. Then $\hat\lambda$ is geometrically simple.
\end{lemma}

\begin{proof}
Let $u_i\in\mathrm{int}(\sP)$, $i=1,2,$ be two eigenvectors corresponding to $\hat\lambda$. Suppose that $u_1\notin \mathrm{span}\{u_2\}$.
Consider $\alpha>0$ such that $u_1-\alpha u_2\in \partial\sP\setminus\{0\}$. Since $\bI u_i=-\lambda^{-1} u_i$ for $i=1,2,$ it is easily seen that
$$\sM^-_*(u_1-\alpha u_2)\leq \bI u_1-\alpha\, \bI u_2= -\lambda^{-1}(u_1-\alpha u_2)\leq 0.$$
Since $u_1-\alpha u_2\geq 0$ in $\Rd$, using Theorem~\ref{T2.4}, we obtain that $u_1-\alpha u_2>0$ in $\cD$ and for some $\kappa>0$
$$\inf_{\cD}\; \tfrac{u_1-\alpha u_2}{\delta^s}\geq \kappa.$$
Then repeating the arguments of Lemma~\ref{L2.1} we could should that $u_1-\alpha u_2\in\mathrm{int}(\sP)$. This is a contradiction. Hence the proof.
\end{proof}

Combining Lemma~\ref{L2.1}, ~\ref{L2.2} and ~\ref{L2.3} we have the following result.
\begin{theorem}\label{T3.2}
There exists functions $\Psi^+, \Psi^-\in\cC^{2s+}(\cD)\cap\cC_0(\cD)$ such that
\[\left\{ \begin{array}{lll}
\bI \Psi^+ = - \lambda^+\Psi & \text{in}\; \cD,
\\
\Psi^+ > 0 & \text{in}\; \cD,
\\
\Psi^+=0 & \text{in}\; \cD^c,
\end{array}
\right.
\quad and \quad
\left\{ \begin{array}{lll}
\bI \Psi^- = - \lambda^-\Psi & \text{in}\; \cD,
\\
\Psi^- < 0 & \text{in}\; \cD,
\\
\Psi^- = 0 & \text{in}\; \cD^c,
\end{array}
\right.
\]
where $\lambda^{\pm}$ are positive. Moreover, $\lambda^+$ (resp. $\lambda^-$) is the only eigenvalue corresponding to a positive (resp. negative) eigenfunction. Also both the eigenvalues
are geometrically simple.
\end{theorem}

\begin{theorem}\label{T3.3}
It holds that $\lambda^+=\widehat\Lambda^+= \Lambda^+$ and $\lambda^-=\widehat\Lambda^-= \Lambda^-$.
\end{theorem}

\begin{proof}
We  prove that $\lambda^+=\widehat\Lambda^+= \Lambda^+$ and the proof for the second one is analogous. It is clear that
$\lambda^+\leq \widehat\Lambda^+\leq \Lambda^+$. Suppose that $\lambda^+<\Lambda^+$. Then we could find $\lambda \in (\lambda^+, \Lambda^+)$ 
with 
$$\bI \psi + \lambda\psi\leq 0\quad \text{in}\; \cD, \quad \psi\geq 0, \quad \text{and}\quad \psi>0\; \text{in}\; \cD.$$
Since $\sM^-_*\psi\leq 0$ it follows from the proof of Theorem~\ref{T2.4} that for some $\kappa>0$ we have
\begin{equation}\label{ET2.3A}
\inf_{\cD} \tfrac{\psi}{\delta^s}\geq \kappa.
\end{equation}
Define $t_0=\sup\{t\in(0, \infty)\; :\; \psi-t\Psi^+>0 \; \text{in}\; \cD\}$. Since $\Psi^+>0$ in $\cD$ we have $t_0<\infty$.
From \eqref{ET1.6A} and \eqref{ET2.3A} it is also easily seen that $t_0$ is positive. We claim that
$\Phi=\psi-t_0\Psi^+$ vanishes at some point in $\cD$. If not, then we much have $\Phi>0$ in $\cD$. Again from \cite[Theorem~5.9]{CS09}
we note that
$$\sM^+_*(-\Phi)\geq -\lambda^+ t_0\Psi^+ + \lambda\psi=\lambda^+\Phi+ (\lambda-\lambda^+)\psi>0\quad \text{in}\; \cD,$$
and hence,
\begin{equation}\label{ET2.3B}
\sM^-_*(\Phi)\leq 0 \quad \text{in}\; \cD, \quad \text{and}\; \; \Phi\geq 0\; \text{in}\; \Rd.
\end{equation}
Again using Theorem~\ref{T2.4}  we find $\kappa_1>0$ satisfying
$$\Phi(x)\geq \kappa_1 \delta^s(x)\quad \text{for}\; x\in\cD.$$
Since $\Psi^+(x)\leq \kappa_2 \delta^s(x)$ for some $\kappa_2>0$, by \eqref{ET1.6A}, we obtain that
$$\Phi(x)> \frac{\kappa_1}{2\kappa_2}\Psi^+(x)\Rightarrow \psi-(t_0+\frac{\kappa_1}{2\kappa_2})\Psi^+>0 \quad \text{in}\; \cD.$$
This contradicts the definition of $t_0$. Thus we must have that for some $x_0\in\cD$, $\Phi(x_0)=0$. But then using Theorem~\ref{T2.4} and \eqref{ET2.3B} we obtain
$\psi=t_0\Psi^+$ in $\bar\cD$. Note that $0\geq \bI\psi + \lambda\psi\geq \bI\Psi^+ + \lambda\Psi^+=(\lambda-\lambda^+)\Psi^+$. But this is a
contradiction to that fact $\lambda>\lambda^+$. This proves that $\lambda^+=\widehat\Lambda^+= \Lambda^+$.
\end{proof}

\subsection{Proofs of Theorems~\ref{T1.1}-\ref{T1.8}}

\begin{proof}[Proof of Theorem~\ref{T1.1}] Proof follows from Theorem~\ref{T3.2} and ~\ref{T3.3}.\end{proof}

\begin{proof}[Proof of Theorem~\ref{T1.2}]
We consider \eqref{ET2.4B} first. Since $\bI u\geq -\Lambda^+\norm{u}_{L^\infty(\cD)}$ in $\cD$. Then using the upper barrier in Lemma~\ref{RS-Lem1} it is easily seen that
\begin{equation}\label{ET2.4C}
u(x)\leq \kappa_1\delta^s(x)\quad  \text{for}\; x\in\cD. 
\end{equation}
As before we again consider $t_0=\sup\{t\; :\; \Psi^+-t u>0\; \; \text{in}\; \cD\}$. Since $u(x_0)>0$ it follows that $t_0<\infty$. Again from
\eqref{ET2.4C} we obtain that $t_0\in(0, \infty)$. Thus from the arguments of Theorem~\ref{T3.3} we would find $x_1\in\cD$ satisfying
$(\Psi^+-t_0 u)(x_1)=0$.  Since $\sM^-_*(\Psi^+-t_0 u)\leq 0$ in $\cD$, it follows from the proof of Theorem~\ref{T2.4} that $\Psi^+=t_0 u$
in $\cD$ implying $u \in \cC^{2s+}(\cD)$. Again applying $\sM^-_*(\Psi^+-t_0 u)\leq 0$ we obtain that $\Psi^+=t_0 u$ in $\Rd$.

Now we consider \eqref{ET2.4A}. If $u^+>0$ in $\cD$ then the result follows from the second part. So assume that $u\leq 0$. This of course, implies $\sM^-_*(-u)\leq 0$ in $\cD$. Therefore,
by Theorem~\ref{T2.4} we have either $u=0$ in $\Rd$ or $u<0$ in $\cD$. We only need to consider the second case. By Theorem~\ref{T1.1} we obtain that $\Lambda^+=\Lambda^-$ and
$u=\kappa\Psi^-$ for some $\kappa>0$. From convexity of $\bI$ we find that
$$0=\bI(\Psi^- - \Psi^-)\leq \bI\Psi^- + \bI(-\Psi^-)=\Lambda^+(-\Psi^-) + \bI(-\Psi^-),$$
and thus, applying the first part of the proof we get $\kappa_1>0$ with $-\Psi^-=\kappa_1\Psi^+$. Hence $u=-\kappa\, \kappa_1\Psi^+$.
\end{proof}

\begin{proof}[Proof of Theorem~\ref{T1.4}]
(a) follows from the first part of the proof of Theorem~\ref{T1.2}. Note that $\bI\Psi^+ + \mu\Psi^+\leq 0$ in $\cD$.
If $u(x_0)>0$ for some $x_0\in\cD$, then the proof of Theorem~\ref{T1.2} gives that $u=t\Psi^+$ for some $t>0$. This would imply 
$\lambda=\lambda^+$ which contradicts our assertion. Thus we must have $u\leq 0$ in $\Rd$. Similarly, one can also prove (b).
\end{proof}

\begin{proof}[Proof of Theorem~\ref{T1.3}]
If $\mu<\Lambda^+$ then the claim follows from Theorem~\ref{T1.4}. Now suppose that $\mu\in(\Lambda^+, \Lambda^-)$. By Theorem~\ref{T1.4}(b) we get that
 $u\geq 0$ in $\Rd$.
In fact, by Theorem~\ref{T2.4}, we have $u>0$ in $\cD$, since $u\neq 0$. Thus by Theorem~\ref{T1.1} we have $u=t \Psi^+$ for $t>0$ and therefore, $\mu=\Lambda^+$. This is a contrdiction.
Therefore, $u$ must be $0$ for $\mu\in(\Lambda^+, \Lambda^-)$. Thus we consider the case $\mu\in (\Lambda^-, \Lambda^-+\varepsilon_0)$ and we show existence of a suitable $\varepsilon_0$.
On the contrary, suppose that no such $\varepsilon_0$ exists. Then we can find a sequence of $\mu_n\downarrow\Lambda^-$ and $u_n\neq 0$ satisfying
\begin{equation}\label{ET2.6B}
\begin{split}
\bI u_n + \mu_n u_n &=0\quad \text{in}\; \cD,
\\
u_n &= 0 \quad \text{in}\; \cD^c.
\end{split}
\end{equation}
We may normalize $u_n$ so that $\norm{u_n}_{L^\infty(\Rd)}=1$. Thus by Theorem~\ref{T2.6} we get that $\sup_n\norm{u_n}_{\cC^\alpha(\Rd)}<\infty$. Therefore, 
by Arzel\`{a}-Ascoli and a we get $u\in\cC(\Rd)$ with $\norm{u_n-u}_{L^\infty(\Rd)}\to 0$ as $n\to\infty$. Moreover, $\norm{u}=1$ and letting
$n\to\infty$ in \eqref{ET2.6B} we get
\begin{equation}\label{ET2.6C}
\begin{split}
\bI u + \lambda^- u &=0\quad \text{in}\; \cD,
\\
u &= 0 \quad \text{in}\; \cD^c.
\end{split}
\end{equation}
By Theorem~\ref{T1.2} and \eqref{ET2.6C} we obtain that $u=t\Psi^-$ for some $t\neq 0$. We also note that by Theorem~\ref{RS-Thm1} we can extract the subsequence above in such way that
\begin{equation}\label{ET2.6D}
\sup_{\bar\cD}\left|\tfrac{u_n}{\delta^s}- \tfrac{u}{\delta^s}\right|\to 0, \quad \text{as}\quad n\to\infty.
\end{equation}
Now suppose that $t>0$. Using \eqref{ET1.6A} we find $\kappa>0$ such that $-\Psi^-\geq \kappa \delta^s$ in $\cD$. Then using \eqref{ET2.6D}, we find that for all large $n$ 
$$u_n(x)\leq t\Psi^-(x) + \frac{t\kappa}{2} \delta^s(x)\leq - t\kappa \delta^s(x)  + \frac{t\kappa}{2} \delta^s(x)=- \frac{t\kappa}{2} \delta^s(x)<0.$$
But this is not possible, by Theorem~\ref{T1.1}. Now suppose that $t<0$. Then it must hold that $\Lambda^-=\Lambda^+$ and $-\Psi^-=\kappa_2\Psi^+$ for some $\kappa_2>0$.
A similar argument as above would give that $u_n>0$ in $\cD$ for all large $n$ which is again not possible. Then $t=0$ is the only choice and this is contradicting to
the fact that $\norm{u}_{L^\infty(\Rd)}=1$. Therefore, we can find $\varepsilon_0>0$ as claimed in the theorem.
\end{proof}

\begin{proof}[Proof of Theorem~\ref{T1.5}]
Suppose that $\Lambda=\Lambda^+$. As earlier, we consider $t_0=\sup\{t\; :\; u-t\Psi^+>0 \quad \text{in}\; \cD\}$. Then apply the arguments of Theorem~\ref{T1.2} to conclude that $u=t_0\Psi^+$.
Similar argument works for the case $\lambda=\Lambda^-$.
\end{proof}

\begin{proof}[Proof of Theorem~\ref{T1.6}]
We only prove $\Lambda^+(\cD_n)\to \Lambda^+(\cD)$ and the other proof is analogous. Denote by $\Lambda^+_n=\Lambda^+(\cD_n)$.
 The proof is divided in three steps.

{\bf Step 1.}\, Suppose that $\cD_n\Subset\cD_{n+1}\Subset \cD$ for all $n$. Then by Corollary~\ref{C2.2} we get $\Lambda^+(\cD_n)>\Lambda^+(\cD_{n+1})>\Lambda^+(\cD)$.
Let $\mu^+=\lim_{n\to\infty}\Lambda^+(\cD_n)$. We consider the solution $w$ of the following equation.
\[
\begin{split}
\bI w  &=-1\quad \text{in}\; \cD,
\\
w &= 0 \quad \text{in}\; \cD^c.
\end{split}
\]
The existence is guaranteed by Theorem~\ref{CS-Thm2}. Since $\bI(0)=0>-1$ by comparison principal we have $w\geq 0$.
Moreover, by Theorem~\ref{T2.4} we note that $w>0$ in $\cD$. Let $(\Psi^+_n, \Lambda^+_n)$ be the eigenpair
satisfying
\begin{equation}\label{ET2.9A}
\bI \Psi^+_n + \Lambda^+_n \Psi^+_n =0 \quad \text{in}\; \cD_n, \quad \text{and}\; \Psi^+=0\quad \text{in}\; \cD^c_n.
\end{equation}
After normalization we may assume that $\norm{\Psi^+_n}_{L^\infty(\Rd)}=1$. 
Then it follows from \eqref{ET2.9A} that $\bI \Psi^+_n\geq -\Lambda^+_n$ in $\cD_n$. Hence by comparison principle
(Theorem~\ref{CS-Thm1}) we obtain that
\begin{equation}\label{ET2.9B}
0\leq \Psi^+_n(x)\leq \Lambda^+_1 w(x) \quad \text{for all}\; x\in\Rd, \quad n\geq 1.
\end{equation}
We claim that there is a subsequence $\{n_k\}$ and $\varphi\in\cC(\Rd)$ satisfying
\begin{equation}\label{ET2.9C}
\sup_{x\in\Rd}\abs{\Psi^+_{n_k}(x)-\varphi(x)}\to 0 \quad \text{as}\quad n_k\to\infty.
\end{equation}
Fix any compact $K\subset\cD$. Then if we choose $n$ sufficiently large so that $K\Subset \cD_n$, from \cite[Theorem~4.2]{K13} (see also \eqref{ET1.6B}), we have
$\norm{\Psi^+_n}_{\cC^1(K)}<\kappa$, for some constant $\kappa$ that does not depend on $n$ for all $n$ large. Therefore, employing a usual diagonalization
process we can extract a subsequence $\{\Psi^+_{n_k}\}$ and a function $\varphi\in \cC(\cD)$ such that 
$$\Psi^+_{n_k}\to \varphi \quad \text{in}\; \cC_{\mathrm{loc}}(\cD), \quad \text{as}\; n_k\to\infty.$$
Next we show that this subsequence $\{n_k\}$ works for our claim \eqref{ET2.9C}. Since $0\leq \varphi\leq \Lambda^+_1 w$ in $\cD$, by \eqref{ET2.9B}, we extend 
$\varphi$ in $\Rd$ by setting its value $0$ on the complement of $\cD$ and we get $\varphi\in\cC(\Rd)$.
Pick $\varepsilon>0$. Choose $n_{k_0}$ large so that 
$$0\leq (\Psi^{+}_{n_k}+\varphi)(x)<\varepsilon \quad \cD^c_{n_{k_0}}, \quad \text{for}\; n_k\geq n_{k_0}.$$
This is possible due to \eqref{ET2.9B}. Therefore, using the local convergence of $\{\Psi^+_{n_k}\}$, we can find $n_{k_1}\geq n_{k_0}$ so that
$$\sup_{\Rd}\abs{\Psi^+_{n_k}-\varphi}=\sup_{\cD}\abs{\Psi^+_{n_k}-\varphi}< 2\varepsilon  \quad \text{for}\; n_k\geq n_{k_1}.$$
This proves the claim \eqref{ET2.9A}. This of course, implies $\norm{\varphi}_{L^\infty(\Rd)}=1$ and $\varphi\gneq 0$. 
Moreover, by the stability of viscosity solution \cite[Corollary~4.7]{CS09}
we have
\begin{equation}\label{ET2.9D}
\bI \varphi + \mu^+ \varphi =0 \quad \text{in}\; \cD, \quad \text{and}\; \varphi=0\quad \text{in}\; \cD^c.
\end{equation}
By \eqref{ET2.9D} and Theorem~\ref{T2.4} it follows that $\varphi>0$ in $\cD$. Then by Theorem~\ref{T1.1} we have $\mu^+=\Lambda^+(\cD)$.

{\bf Step 2.} \, Suppose that $\cD_n\Supset\cD_{n+1}\Supset \cD$ for all $n$.
As before, we consider the eigenpair $(\Psi^+_n, \lambda^+_n)$ with $\norm{\Psi^+_n}_\infty=1$. Let $\mu^+=\lim_{n\to\infty} \Lambda^+_n$ and existence of the limit
is assured by the monotonicity of  $\{\Lambda^+_n\}$.

Let $\cD_\varepsilon=\{x\in \Rd : \dist(x, \cD)<\varepsilon\}$. Since $\cD$ is smooth, we will have $\cD_\varepsilon$ smooth for $\varepsilon$
sufficiently small. Moreover, we can fix a $r_0, \varepsilon_1>0$ such a way that for every $\varepsilon<\varepsilon_1$, 
$\cD_\varepsilon$ satisfies uniform exterior sphere condition
with radius $r_0$. For every $\varepsilon$, let $w_\varepsilon, w_\varepsilon>0$ in $\cD_\varepsilon$, be the solution of
\[
\begin{split}
\bI w_\varepsilon  &=-1\quad \text{in}\; \cD_\varepsilon,
\\
w_\varepsilon &= 0 \quad \text{in}\; \cD^c_\varepsilon.
\end{split}
\]
Then it follows from that proof of \eqref{ET1.6A} that
\begin{equation}\label{ET2.9E}
w_\varepsilon\leq \kappa_2\, \delta^s_\varepsilon(x) \quad x\in\cD_\varepsilon, \quad \varepsilon<\varepsilon_1,
\end{equation}
for some constant $\kappa_2$, not depending on $\varepsilon$, where $\delta_\varepsilon$ denotes the distance function from the boundary of $\cD_\varepsilon$.
Now using a similar argument, as in step 1, we can find a subsequence ${\Psi_{n_k}}$ and $\varphi\in\cC(\cD)$ such that
$$\Psi^+_{n_k}\to \varphi \quad \text{in}\; \cC_{\mathrm{loc}}(\cD), \quad \text{as}\; n_k\to\infty.$$
Since for every fixed $\varepsilon$, we shall have $n_k$ large satisfying $\cD_{n_k}\Subset\cD_\varepsilon$ it follows from \eqref{ET2.9E} that 
$\varphi(x)\leq \Lambda^+(\cD) w_\varepsilon(x)\leq \Lambda^+(\cD)\kappa_2\, \delta^s_\varepsilon(x)$
for every $x\in \cD$. Hence if we set $\varphi=0$ in $\cD^c$ we still have $\varphi\in\cC(\Rd)$. Moreover, \eqref{ET2.9C} holds. Then the rest of argument
can be mimicked from step 1 to show $\mu^+=\Lambda^+(\cD)$.

{\bf Step 3.} Now we consider the case of when the domains are not necessarily ordered. Since $\cD_n\to\cD$ we can find two sequences of smooth domains
$\{\cD'_n\}$ and $\{\cD^{\prime\prime}_n\}$ such that $\cD'_n$ is increasing to $\cD$, $\cD^{\prime\prime}_n$ is decreasing to $\cD$ and 
$$ \cD'_n\Subset \cD_n\Subset\cD^{\prime\prime}_n \quad \text{for all}\; n\geq 1.$$
Note that $\Lambda^+(\cD'_n)> \Lambda^+(\cD_n)> \Lambda^+(\cD^{\prime\prime}_n)$, by Corollary~\ref{C2.2}. Then the proof follows from step 1 and step 2 above.
\end{proof}

\begin{proof}[Proof of Theorem~\ref{T1.7}]
The uniqueness of solution of \eqref{ET2.10A} follows from Corollary~\ref{C2.1}. We only need to show existence of solution. We use Schauder fixed point theory to do so.
For $\psi\in\cC_0(\cD)$, let $\varphi=\sG[\psi]\in\cC_0(\cD)$ be the unique viscosity solution of
$$\bI\varphi=f-\mu \psi \quad \text{in}\; \cD, \quad \text{and}\quad \varphi=0\quad \text{in}\; \cD^c.$$
That is,
\begin{equation}\label{ET2.10B}
\bI(\sG[\psi])=f-\mu \psi \quad \text{in}\; \cD, \quad \text{and}\quad \sG[\psi]=0\quad \text{in}\; \cD^c.
\end{equation}
Thus any fixed point of $\sG$ will be a solution to \eqref{ET2.10A}. We show the following
\begin{itemize}
\item[(a)] $\sG$ is continuous and compact.
\item[(b)] The set $\{\varphi\in \cC_0(\cD)\; :\; \varphi=\kappa\sG[\varphi]\; \text{for some }\, \kappa\in[0, 1]\}$ is bounded in $\cC_0(\cD)$.
\end{itemize}
Once we have shown (a) and (b), the existence follows from the Schauder fixed point theorem.

Suppose that $\{\psi_n\}$ be a bounded sequence in $\cC_0(\cD)$. Then using Theorem~\ref{T2.6} we get that $\{\sG[\psi_n]\}$ is a bounded sequence in 
$\cC^\alpha(\Rd)$. Thus, it is a compact sequence in $\cC_0(\cD)$. Therefore, $\sG$ is a compact operator. Continuity of $\sG$ follows from compactness,
stability property of viscosity solution and  the comparison principle (Theorem~\ref{CS-Thm1}).
This proves (a).

To prove (b) we consider the set
$$\cB=\{\varphi\in \cC_0(\cD)\; :\; \varphi=\kappa\sG[\varphi]\; \text{for some }\, \kappa\in[0, 1]\}.$$
From \eqref{ET2.10B} we see that for $\varphi\in\cB$ and $\kappa\in[0,1]$
\begin{equation}\label{ET2.10C}
\bI\varphi+\mu\kappa\, \varphi= \kappa f \quad \text{in}\; \cD, \quad \text{and}\quad \varphi=0\quad \text{in}\; \cD^c.
\end{equation}
Using Theorem~\ref{T1.6}, we can choose $\cD_1\Supset\cD$ such that $\Lambda^+(\cD_1)=\mu_1>\mu$. This is possible since $\mu<\Lambda^+$. Let 
$\Psi^+_1$ be the positive eigenfunction corresponding to $\mu_1$ i.e., 
\begin{equation}\label{ET2.10D}
\bI\Psi^+_1+\mu_1 \Psi^+_1= 0 \quad \text{in}\; \cD_1, \quad \text{and}\quad \Psi^+_1=0\quad \text{in}\; \cD^c_1.
\end{equation}
Since $\Psi^+_1>0$ in $\bar\cD$, we can multiply $\Psi^+$ with a constant, dependent only on $\norm{f}_\infty$, so that \eqref{ET2.10D} gives us
$$\bI\Psi^+_1+\mu\kappa\, \Psi^+_1\leq -\norm{f}_\infty \quad \text{in}\; \cD, \quad \text{and}\quad \Psi^+_1\geq 0\quad \text{in}\; \cD^c,$$
for all $\kappa\in[0, 1]$. Hence using \eqref{ET2.10C} and Corollary~\ref{C2.1} we obtain $\varphi\leq \Psi^+$. Since $\bI(\varphi)+\bI(-\varphi)\geq 0$, by convexity, it
is easily seen from \eqref{ET2.10C} that
$$\bI(-\varphi)+\mu\kappa\, (-\varphi)\geq - \kappa f \quad \text{in}\; \cD, \quad \text{and}\quad \varphi=0\quad \text{in}\; \cD^c.$$
Therefore, we can again apply Corollary~\ref{C2.1} to get $-\varphi\leq \Psi^+$. Combining we have 
$\norm{\varphi}_{L^\infty(\cD)}\leq \norm{\Psi^+}_{L^\infty(\Rd)}$ for all $\kappa\in(0, 1]$
and any $\varphi\in\cB$. This proves (b).
\end{proof}

\begin{proof}[Proof of Theorem~\ref{T1.8}]
To prove it by  contradiction we assume that there exists a sequence $(\mu_n, u_n)$ satisfying \eqref{ET2.7A}, and $\mu_n\downarrow \Lambda^-$ as $n\to\infty$
and $u_n\nless 0$ in $\cD$ for all $n$.
First we note that $\inf_n\norm{u_n}_\infty>0$, otherwise passing to the limit we would get a viscosity solution $0$ of the equation \eqref{ET2.7A}, and this would contradict the non-zero
property of $f$. Since $\bI u_n + \mu u_n\leq 0$, it is follows from Theorem~\ref{T1.1} that $u_n$ can not be non-negative in $\cD$. Therefore, $u_n$ must change sign in $\cD$. 

Suppose that $\sup\norm{u_n}_\infty<\infty$. Using Theorem~\ref{T2.6}, we can find a subsequence, denote by $u_n$, such that $\norm{u_n-u}_\infty\to 0$ as $n\to\infty$.
Moreover,
\begin{equation}\label{ET2.7B}
\begin{split}
\bI u + \Lambda^- u &=f\quad \text{in}\; \cD,
\\
u &= 0 \quad \text{in}\; \cD^c.
\end{split}
\end{equation}
If $u(x_0)<0$ for some $x_0\in\cD$, by Theorem~\ref{T1.2} and \eqref{ET2.7B}, we get that $u=t\Psi^-$ for some $t>0$. But this is not possible as $f\neq 0$. Thus we must have $u\gneq 0$.
Again by Hopf's lemma this would imply $u>0$ in $\cD$ and $\Lambda^+=\Lambda^-$. Then it follows fromTheorem~\ref{T1.5} that $u=t\Psi^+$ for some $t>0$. This would again contradict 
the fact that $f\neq 0$.

Thus it must hold that for some sub-sequenec, denoted by $\{u_n\}$, $\lim_{n\to\infty}\norm{u_n}_\infty=+\infty$. Define 
$\tilde{u}_n=\norm{u_n}^{-1}_\infty u_n$.Then using Theorem~\ref{RS-Thm1} and ~\ref{T2.6} we can extract a subsequence
$\{\tilde{u}_n\}$  satisfying
\begin{equation}\label{ET2.7D}
\sup_{x\in\bar\cD}\left|\tfrac{\tilde{u}_n(x)}{\delta^s(x)}-\tfrac{\tilde{u}(x)}{\delta^s(x)}\right|\to 0, \quad \text{as}\; n\to\infty.
\end{equation}
Moreover, by stability of viscosity solution we get
\begin{equation}\label{ET2.7C}
\begin{split}
\bI \tilde{u} + \Lambda^- \tilde{u} &=0\quad \text{in}\; \cD,
\\
\tilde{u} &= 0 \quad \text{in}\; \cD^c.
\end{split}
\end{equation}
Since $\tilde{u}\neq 0$, using \eqref{ET2.7C} and Theorem~\ref{T1.2} we either have $\tilde{u}=t\Psi^-$ for some positive $t$ or
$\Lambda^+=\Lambda^-$ and $\tilde{u}=t\Psi^+$ for some positive $t$. Hence using Theorem~\ref{T2.4} we can find $\kappa_3>0$ such
that $\delta^{-s}\tilde{u}\leq -\kappa_3$ in $\cD$, in the first case, or $\delta^{-s}\tilde{u}\geq \kappa_3$ in $\cD$, in the second case. Combining this
with \eqref{ET2.7D} we either get $u_n<0$ or $u_n>0$ in $\cD$ for all large $n$. Since $f\leq 0$, the second possibility would imply
 $$\bI u_n + \mu_n u_n \leq 0\quad \text{in}\; \cD, \quad \mu_n>\Lambda^-\geq\Lambda^+,$$
 which is not possible. Again $u_n<0$ in $\cD$ is contradicting to our hypotheses on the sequence.
This proves the theorem.
\end{proof}

Rest of the section is devoted to the proof of Theorem~\ref{T1.9}. We shall closely follow the approach of \cite{BEQ}. To do show we consider $\sM^+_*$ as a function
of $\lambda$ (see \eqref{Assmp1}). Fix $\Lambda>0$. Denote by $\sM^+_{\lambda}=\sM^+_*$ where the supremum is taken over all the kernel satisfying \eqref{Assmp1} for 
$0<\lambda\leq \Lambda$. Let $\Lambda^\pm(\lambda)$ be the principal eigenvalues of $\sM^+_{\lambda}$. Then it is easily seen that for $\lambda_1\leq \lambda_2\leq \Lambda$
we have $\Lambda^+(\lambda_1)\leq \Lambda^+(\lambda_2)$ and $\Lambda^-(\lambda_1)\geq \Lambda^-(\lambda_2)$. Let us now
establish Lemma~\ref{L3.4} below  . These lemmas should be compared with \cite[Lemma~2.3, 2.4, Prop.~2.3]{BEQ} which are done for
(local) elliptic Pucci's operator.

\begin{lemma}\label{L3.4}
The following hold.
\begin{itemize}
\item[(i)] The functions $\Lambda^+, \Lambda^-:(0, \Lambda]\to \RR_+$ are continuous.

\item[(ii)] For an interval $[a, b]\subset (0, \Lambda]$ there exists $\eta>0$ such that for all $\lambda\in[a, b]$ there is no solution to
\begin{equation}\label{EL3.4A}
\begin{split}
\sM^+_\lambda u + \mu u &=0\quad \text{in}\; \cD,
\\
u &= 0 \quad \text{in}\; \cD^c,
\end{split}
\end{equation}
for $\mu\in(\Lambda^-(\lambda), \Lambda^-(\lambda)+\eta]$.
\end{itemize}
\end{lemma}

\begin{proof}
We begin with the proof of (i). We only prove the claim for $\Lambda^+$, and the proof for $\Lambda^-$ is analogous. Let $\lambda_n\to \lambda_0\in(0, \Lambda]$. Choose
$\hat\lambda>0$ in such a way that $\{\lambda_n\}\subset[\hat\lambda, \Lambda]$ holds. Note that the class $\cL_n$ of all operators satisfying 
\eqref{Assmp1} with parameters $\lambda_n, \Lambda$ is actually a subclass of the operators satisfying \eqref{Assmp1} with parameter $\hat\lambda, \Lambda$.
Thus the stability constants appeared in the results of section~\ref{S-prelim} could be choosen independent of $n$ and dependent on $\hat\lambda, \Lambda$.
Let $(\Psi^+(\lambda_n), \Lambda^+(\lambda_n))$ be the positive principal eigenpair and $\norm{\Psi^+(\lambda_n)}_{L^\infty(\cD)}=1$. 
Using Theorem~\ref{T2.6} and the fact $\Lambda^+(\lambda_n)\leq \Lambda^+(\Lambda)$
we can extract a subsequence, denoted by $\{(\Psi^+(\lambda_n), \Lambda^+(\lambda_n))\}$, such that 
$$\Psi^+(\lambda_n)\to \Psi\quad \text{in}\; \cC_0(\Rd), \quad \Lambda^+(\lambda_n)\to \tilde\Lambda, \quad \text{as}\; n\to\infty.$$
From the stability of viscosity solution it is easily seen that 
$$\sM^+_{\lambda_0}\Psi+\tilde\Lambda \Psi=0, \quad \text{in}\; \cD, \quad \text{and}\quad \Psi=0\quad \text{in}\; \cD^c.$$
Moreover, $\Psi\gneq 0$. Using Theorem~\ref{T2.4} we note that $\Psi>0$ in $\cD$ and therefore, by Theorem~\ref{T1.1} we obtain $\tilde\Lambda=\Lambda^+(\lambda_0)$.
This establishes the continuity of $\Lambda^+$.

Now we show (ii). Suppose that the claim is not true. Then there exists a sequence $\{\lambda_n\}\subset [a, b]$, $\{\mu_n\}$ such that 
$\lambda_n\to\bar\lambda\in[a, b]$, 
$\mu_n>\Lambda^-(\lambda_n)$ and $(\mu_n-\Lambda^-(\lambda_n))\to 0$ as $n\to\infty$. Moreover, there exists $u_n\neq 0$ satisfying 
\begin{equation}\label{EL3.4B}
\begin{split}
\sM^+_{\lambda_n} u_n + \mu_n u_n &=0\quad \text{in}\; \cD,
\\
u_n &= 0 \quad \text{in}\; \cD^c.
\end{split}
\end{equation}
From (i) we also note that $(\mu_n-\Lambda^-(\bar\lambda))\to 0$, as $n\to\infty$. Using Theorem~\ref{RS-Thm1} we can extract a subsequence, denoted by $\{u_n\}$, satisfying
 $$\sup_{\cD}\left|\frac{u_n}{\delta^s}-\frac{u}{\delta^s}\right|\to 0, \quad \text{as}\; n\to\infty,$$
 for some $u\in\cC_0(\Rd)\setminus\{0\}$. Moreover, we also have from \eqref{EL3.4B} that
 \begin{equation*}
\begin{split}
\sM^+_{\bar\lambda} u + \Lambda^-(\bar\lambda) u &=0\quad \text{in}\; \cD,
\\
u &= 0 \quad \text{in}\; \cD^c.
\end{split}
\end{equation*}
Therefore, following the arguments of Theorem~\ref{T1.3} we get a contradiction. Hence the proof.
\end{proof}

Let $\cI^+_\lambda$ be the inverse operator of $-\sM^+_\lambda$ i.e.
$$-\sM^+_\lambda (\cI^+_\lambda \varphi)=\varphi\quad \text{for}\; \varphi\in\cC_0(\cD).$$
From Lemma~\ref{L2.1} we note that $\cI^+_\lambda$ is compact for every $\lambda\in (0,\Lambda]$. 
Define
$$\Lambda_2(\lambda)=\inf\{\mu \; :\; \text{there is a non-zero solution to}\, \eqref{EL3.4A}\}.$$
By $\sB(0, r)$ we denote the ball of radius $r$ around
the $0$ function in $\cC_0(\cD)$.

\begin{lemma}\label{L3.5}
Let $r, \lambda>0$ and $\mu\in\RR$. Then
\[
\Deg_{\cC_0(\cD)}(I-\mu\cI^+_\lambda, \sB(0, r), 0)=
\left\{\begin{array}{lll}
1 & \text{if}\; \mu<\Lambda^+(\lambda),
\\[2mm]
0 & \text{if}\; \Lambda^+(\lambda)<\mu<\Lambda^-(\lambda),
\\[2mm]
-1 & \text{if}\; \Lambda^-(\lambda)<\mu<\Lambda_2(\lambda).
\end{array}
\right.
\]
\end{lemma}

\begin{proof}
For $\mu<\Lambda^+(\lambda)$ and $\Lambda^-(\lambda)<\mu<\Lambda_2(\lambda)$ the proof follows from the invariance of the degree under homotopy
and Lemma~\ref{L3.4}. This can be done following the arguments in \cite[Proposition~2.3]{BEQ}. Now suppose $\Lambda^+(\lambda)<\mu<\Lambda^-(\lambda)$.
To prove by contradiction let us assume that $\Deg_{\cC_0(\cD)}(I-\mu\cI^+_\lambda, \sB(0, r), 0)\neq 0$. Thus we can find $\kappa>0$ such that for
any $f$ with $\norm{f}_{\cC_0(\cD)}\leq \kappa$ we must have $\Deg_{\cC_0(\cD)}(I-\mu\cI^+_\lambda, \sB(0, r), f)\neq 0$. Let $\kappa_1\in(0, \kappa)$
and $h$ satisfy
\begin{equation}\label{EL3.5A}
\sM^+_\lambda h =-\kappa_1 \quad \text{in}\; \cD, \quad \text{and}\quad h=0\quad \text{in}\; \cD^c.
\end{equation}
This is possible due to Theorem~\ref{CS-Thm2}. Using Theorem~\ref{T2.6}, we may choose $\kappa_1$ small enough so that $\norm{h}_{L^\infty(\cD)}<\kappa$.
Choose $u\neq 0$ such that
\begin{equation}\label{EL3.5B}
u-\mu\, \cI^+_\lambda u = h \quad \text{in}\; \cD, \quad \text{and}\quad h=0\quad \text{in}\; \cD^c.
\end{equation}
From \eqref{EL3.5A} and \eqref{EL3.5B} it is easily seen that
$$\sM^+_\lambda u + \mu u\leq -\kappa_1  \quad \text{in}\; \cD, \quad \text{and}\quad u=0\quad \text{in}\; \cD^c.$$
By Theorem~\ref{T1.4} we get that $u\geq 0$. By Hopf's lemma (Theorem~\ref{T2.4}) we obtain that $u>0$ in $\cD$.
Hence $\mu\leq \Lambda^+(\lambda)$ which is contradicting to the fact  $\Lambda^+(\lambda)<\mu<\Lambda^-(\lambda)$.
Hence the proof.
\end{proof}

\begin{proof}[Proof of Theorem~\ref{T1.9}]
Proof follows along the lines of \cite[Theorem~1.9]{BEQ} using Lemma~\ref{L3.4} and  \ref{L3.5}.
\end{proof}

\section*{Acknowledgements}
The research of Anup Biswas was supported in part by an INSPIRE faculty
fellowship and DST-SERB grant EMR/2016/004810.

%

\end{document}